\documentclass{amsart}
\usepackage{amsfonts,amsmath,amssymb,graphicx,color}
\newtheorem{theorem}{Theorem}
\theoremstyle{plain}

\newtheorem{corollary}{Corollary}

\newtheorem{definition}{Definition}

\newtheorem{lemma}{Lemma}

\newtheorem{remark}{Remark}

\numberwithin{equation}{section}
\begin{document}
\title{The Well-Covered Dimension of  Products of Graphs}

\author[I. Birnbaum]{Isaac Birnbaum}
\address{Department of Mathematics\\
California State University, Fresno. Fresno, CA.}
\email{isaacb1@mail.fresnostate.edu}
\author[M. Kuneli]{Megan Kuneli}
\address{Department of Mathematics\\
California State University, Fresno. Fresno, CA.}
\email{mrkunelius@mail.fresnostate.edu}
\author[R. McDonald]{Robyn McDonald} 
\address{
California State University, Stanislaus. Turlock, CA.}
\email{rmcdonald@csustan.edu}
\author[K. Urabe]{Katherine Urabe} 
\address{Department of Mathematics\\
California State University, Fresno. Fresno, CA.}
\email{kturabe@mail.fresnostate.edu}
\author[O. Vega]{Oscar Vega}
\address{Department of Mathematics\\
California State University, Fresno. Fresno, CA.}
\email{ovega@csufresno.edu}
\thanks{The first four authors were supported by the Undergraduate Research Grant Program and/or the Faculty Sponsored Student Research Awards at California State University, Fresno.}

\begin{abstract}
We discuss how to find the well-covered dimension of a graph that is the Cartesian product of paths, cycles, complete graphs, and other simple graphs.  Also, a bound for the well-covered dimension of $K_n\times G$ is found, provided that $G$ has a largest greedy independent decomposition of length $c<n$.  
 
Formulae to find the well-covered dimension of graphs obtained by vertex blowups on a known graph, and to the lexicographic product of two known graphs are also given.
\end{abstract}

\maketitle

\section{introduction}

In this paper, a graph is understood to be undirected and have no loops or multiple edges. While graphs with multiple edges could be taken under consideration, it is not necessary to do so as multiple edges do not add any difficulty or important properties.

A set of vertices in a graph $G$ is said to be independent if no two vertices in the set are joined by an edge. An independent set $M$ of $G$ is called maximal if no independent set of $G$ properly contains $M$. The largest (in terms of cardinality) maximal independent set (or sets) of $G$ is called a maximum independent set of $G$, and a graph is said to be well-covered if every maximal independent set of $G$ is also maximum. A well-covered graph could also be defined by the property of all maximal independent sets having the same cardinality. This notion was introduced by Plummer in \cite{P}. In \cite{BN}, Brown and Nowakowski defined a well-covered weighting of a graph $G$ as a function $w:V(G)\rightarrow \textbf F$ that assigns values to the vertices of $G$ in such a way that $\sum_{x\in M} w(x)$ is a constant for all maximal independent sets $M$ of $G$. It is immediate from the latter definition that one could re-define well-coveredness by saying that a well-covered graph is a graph that admits the constant function equal to $1$ as a well-covered weighting of $G$.  We will use Brown and Nowakowski's presentation (notation, nomenclature, etc), although this problem was originally introduced by Caro, Ellingham, Ramey, and Yuster in \cite{CER} and \cite{CY}.

It is easy to show that, once a field $\textbf{F}$ is fixed, the set of all well-covered weightings of a graph $G$ is an $\textbf{F}$-vector space, which is called \emph{the well-covered space of $G$}. The dimension of this vector space over $\textbf{F}$ is called the \emph{well-covered dimension of $G$} and is denoted by $wcdim(G,\textbf F)$. If $wcdim(G, \textbf F)$ does not depend on the field used then the well-covered dimension of $G$ is instead denoted as $wcdim(G)$.  Note that $wcdim(G, \textbf F)$ may change depending on $char(\textbf F)$. In \cite{BN},  and later in this article, examples of graphs with variable dimension are discussed. When the characteristic becomes something to consider we will be careful to remark on it.

Our graph theoretic notation, algebraic notation, and matrix theoretic notation are standard; the reader can look at \cite{W} for any concepts we fail to define. The vertex set of a graph $G$ is denoted by $V(G)$. The cardinality of a set of vertices $V$ is denoted by $|V|$.  A field with $q=p^h$ ($p$ prime) elements is denoted by $\textbf F_{q}$. The $n\times n$ identity matrix is denoted by $I_{n}$. The $n\times n$ matrix where each entry is a $1$ is denoted by $J_{n}$. An $m\times 1$ column vector where each entry is a $1$ is denoted by $\textbf 1_{m}$. An $m\times 1$ column vector where each entry is a $0$ is denoted by $\textbf 0_{m}$.

It is relatively simple to calculate the well-covered dimension of a graph $G$, provided $G$ is not too large. One first needs to find all the independent sets of $G$, which can be done using a greedy algorithm. Suppose that the maximal independent sets of $G$ are $M_{i}$ for $i=0,\ldots,k-1$. Then a well-covered weighting $w$ of $G$ is determined by a solution of the linear system of equations formed by selecting a maximal independent set,  in this particular instance $M_0$, and setting the system $M_i - M_0=0$ for $i=1,\ldots,k-1$. This system is homogeneous, and can therefore be written in the form $A\textbf x=\textbf0$. Note that $A$ is an $m\times n$ matrix where $m=k-1$ and $n=|V(G)|$. As this system is homogeneous, the nullity of $A$ (note that $char(\textbf{F})$ could be relevant here) is equivalent to $wcdim(G,\textbf F)$. So, $wcdim(G,\textbf F)=n-rank(A)$. In the case when $n=rank(A)$, then $wcdim(G,\textbf F)=0$, which implies that in this case the only possible well-covered weighting is the $0$ function.

For the remainder of this paper, we shall concern ourselves  only with the determining of the well-covered dimensions for various individual graphs and graph families. We start by recalling a lemma from \cite{BN}, as it will allow us to focus only on connected graphs.

\begin{lemma}[Brown \& Nowakowski \cite{BN}]
Let $G$ and $H$ be graphs. Then 
\[
wcdim(G \cup H, \textbf F) = wcdim(G, \textbf F)+wcdim(H, \textbf F)
\]
\end{lemma}


\bigbreak

Although our main focus is to find the well-covered dimension of products of graphs, we will start with a few general results.

\section{The Well-covered dimension of certain families of graphs}\label{sectionbasic}

The family of complete graphs has the easiest to find well-covered dimension among all (connected) graphs. In fact, by simply looking at the maximal independent sets of $K_n$ we get that $wcdim(K_{n})=1$. Also, only using the technique mentioned above, it is easy to find the well-covered dimension of several families of graphs. In this section we discuss crown graphs, complete multipartite graphs, paths, cycles, and gear graphs.

Recall that, for any $n>2$, the crown graph $S_{n}^{0}$ is formed by removing a perfect matching from $K_{n,n}$. Though not specifically stated as such, it was proven in \cite{BN} that $wcdim(S_{n}^{0},\textbf F)=n-1$, if $char(\textbf F)=0$, and $wcdim(S_{n}^{0},\textbf F)=n$ if both $char(\textbf F)$ and $n$ are even. We shall extend this result to allow us to calculate the well-covered dimensions of all crown graphs over all fields.

\begin{theorem}
Let $S_{n}^{0}$ denote a crown graph, for all $n\in \mathbb{N}$. Then,
\[
wcdim(S_{n}^{0},\textbf F) =
\left\{
\begin{array}{ll}
n & $if $char(\textbf F)= p\neq 0$ and $p| (n-2)$$ \\
n-1 & $otherwise$ 
\end{array}
\right.
\]
\end{theorem}

\begin{proof}
Let $K_{V_{1},V_{2}}$ be the complete bipartite graph with $V_{1}=\{a_{1},\cdots ,a_{n}\}$ and $V_{2}=\{b_{1},\cdots ,b_{n}\}$ and let $a_{1}b_{1},\cdots ,a_{n}b_{n}$ be the perfect matching that is removed from $K_{V_{1},V_{2}}$ to form $S_{n}^{0}$. The maximal independent sets of $S_{n}^{0}$ are $\{a_{i}, b_{i}\}$ for $i=1,2,3,\cdots ,n$, and $V_{1}$ and $V_{2}$. Setting the sum of each of the weights on the maximal independent sets equal to that of the weights on the vertices of $V_{2}$, we find that the linear system corresponding to the well-covered weightings is $A\textbf{x}=0$, where
\[
A=\begin{pmatrix}
I_{n} &  I_{n}-J_{n}\\
\textbf{1}_{n}^{T} & 
-\textbf{1}_{n}^{T}\\
\end{pmatrix},
\]
an $(n+1)\times 2n$ matrix. Subtracting the top $n$ rows from the bottom yields
\[
\begin{pmatrix}
I_{n} &  I_{n}-J_{n}\\
\textbf{0}_{n}^{T} & \left(n-2\right)\textbf{1}_{n}^{T}\\
\end{pmatrix}.
\]
It follows that we have two possibilities depending on whether or not $char(F)$ divides $n-2$. The theorem follows after finding the rank of this matrix in either case.
\end{proof}

\begin{theorem}\label{k-partite}
Let $G=K_{n_{0},\ldots ,n_{k-1}}$ be a complete $k$-partite graph. Then
\[
wcdim(K_{n_{0},...,n_{k-1}}) =\sum_{i=0}^{k-1} n_{i} -(k-1)
\]
\end{theorem}

\begin{proof}
Let $f$ be a well-covered weighting of $G$. We denote the maximal independent sets of $G$ by $N_{i}$, where $|N_i|=n_1$, for all $i=0,\ldots ,k-1$. Setting the sum of each of the weights on the maximal independent sets equal to that of the weights on the vertices of $N_{k-1}$, we find that the linear system corresponding to the well-covered weightings is $A\textbf{x}=0$, where
\[
A=\begin{pmatrix}
\textbf{1}_{n_{0}}^{T} & \textbf{0}_{n_{1}}^{T} & \textbf{0}_{n_{2}}^{T} & \cdots & \textbf{0}_{n_{k-2}}^{T} & -\textbf{1}_{n_{k-1}}^{T}\\
\textbf{0}_{n_{0}}^{T} & \textbf{1}_{n_{1}}^{T} & \textbf{0}_{n_{2}}^{T} & \cdots & \textbf{0}_{n_{k-2}}^{T} & -\textbf{1}_{n_{k-1}}^{T}\\
\textbf{0}_{n_{0}}^{T} & \textbf{0}_{n_{1}}^{T} & \textbf{1}_{n_{2}}^{T} & \cdots & \textbf{0}_{n_{k-2}}^{T} & -\textbf{1}_{n_{k-1}}^{T}\\
\vdots & \vdots & \vdots & \ddots & \vdots & \vdots\\
\textbf{0}_{n_{0}}^{T} & \textbf{0}_{n_{1}}^{T} & \textbf{0}_{n_{2}}^{T} & \cdots & \textbf{0}_{n_{k-2}}^{T} & -\textbf{1}_{n_{k-1}}^{T}\\
\end{pmatrix},
\]
which is a $\displaystyle{(k-1)\times \left(\sum_{i=0}^{k-1} n_{i}\right)}$ matrix. $A$ has rank $k-1$. Hence the nullity is $\displaystyle{\sum_{i=0}^{k-1} n_{i} -(k-1)}$, which is what we wanted to prove.
\end{proof}

\begin{corollary}\label{turan}
Let $T(n,r)$ be a Tur\'an graph. Then
\[
wcdim\left(T\left(n,r\right)\right)=\left(n\bmod r\right)\lceil n/r\rceil+\left(r-\left(n\bmod r\right)\right)\lfloor n/r\rfloor-(r-1)
\]
Moreover, if $r$ divides $n$ then $wcdim\left(T\left(n,r\right)\right)=n-r+1$.
\end{corollary}

The behavior, in terms of well-covered weightings, of paths and cycles is very similar. Hence, we will study these two families simultaneously.

Consider $G$ to be an $n$-path or an $(n+2)$-cycle, for $n\geq 6$. Label six `consecutive' vertices $a,b,c,d,e$ and $f$ as in Figure \ref{sixconsecutivevertices}. Let $w$  be a well-covered weighting of $G$, and let $M_1$ and $M_2$ be two maximal independent sets of $G$ that contain the same vertices, except that $M_1$ contains $\{a,c,f\}$ and $M_2$ contains $\{a,d,f\}$ instead. Locally, these two independent sets are represented in the figure below.
\begin{figure}[ht]
\centering 
\includegraphics[height=.7in]{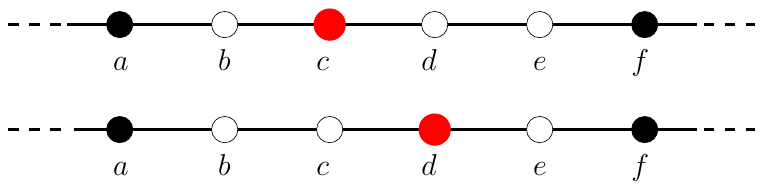}
\caption{$M_1$ and $M_2$ on six consecutive vertices.}
\label{sixconsecutivevertices}
\end{figure}

Since $M_1$ and $M_2$ just `interchange' $c$ and $d$, then $w(c)=w(d)$. It is now immediate that all vertices of $C_n$, for $n\geq 8$, have the same weight for all well-covered weightings of this graph. Hence, $wcdim(C_n) \leq 1$ for all $n\geq 8$.

Now consider two maximal independent sets $N_1$ and $N_2$ of $C_n$, with $n\geq 8$, that contain the same vertices outside of a string of seven consecutive vertices, where $N_1$ and $N_2$ contain four and three vertices respectively. These seven vertices, with the vertices contained in $N_1$ and $N_2$ are represented in Figure \ref{differentnumberofelements} below.
\begin{figure}[ht]
\centering 
\includegraphics[height=.7in]{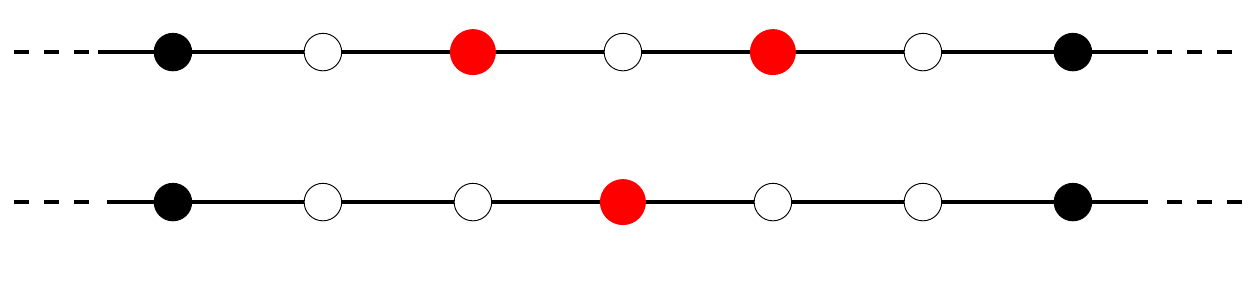}
\caption{Two maximal independent sets with different cardinality.}
\label{differentnumberofelements}
\end{figure}

It follows that this graph admits maximal independent sets with different cardinalities, and thus $wcdim(C_n) =0$ for all $n\geq 8$.

Similarly, from the argument associated to Figure \ref{sixconsecutivevertices}, if $n\geq 6$ and $V(P_n)=\{ v_1, v_2, \cdots , v_n\}$ (edges connecting $v_i$ with $v_{i+1}$) then vertices $v_3, \cdots , v_{n-2}$ must have the same weight for all well-covered weightings of $P_n$. Moreover, for small values of $n$ it is easy to see that these weights must be zero. For larger values of $n$ Figure \ref{differentnumberofelements} provides a way to construct maximal independent sets with different cardinality, which forces $w(v_3) =  \cdots =w(v_{n-2})$.

Finally, we can construct two maximal independent sets of $P_n$ that share all but one vertex, which is $v_1$ for one of them and $v_2$ for the other. This can be seen in the figure below.
\begin{figure}[ht]
\centering 
\includegraphics[height=.7in]{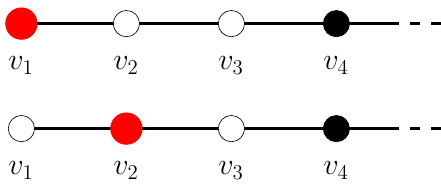}
\caption{Two maximal independent sets on the first four vertices.}
\label{fourendvertices}
\end{figure}

It follows that $w(v_1)=w(v_2)$, and symmetrically that $w(v_{n-1})=w(v_n)$, for all well-covered weightings $w$ of $P_n$. Lastly, we want to remark that $w(v_1)$ is independent of $w(v_n)$, and thus, adding simple computations to the arguments above we obtain the following result:

\begin{theorem}\label{thmpathcycles}
If $w$ is a well-covered weighting of $P_{n}$ and $n\geq 5$, then 
\[
w(v_{1})=w(v_{2}) \hspace{.6in} w(v_{3})=\ldots=w(v_{n-2})=0 \hspace{.6in} w(v_{n-1})=w(v_{n}). 
\]
Moreover,
\[
wcdim(P_n)=
\left\{ \begin{array}{cl}
1 & \text{ if $n=2$} \\
2  & \text{ if $n>2$}
\end{array}
\right. \hspace{.5in} 
wcdim(C_n)=
\left\{ \begin{array}{cl}
3 & \text{ if $n=4$} \\
2 & \text{ if $n=6$} \\  
1 & \text{ if $n=3,5,7$} \\
0  & \text{ if $n\geq8$}
\end{array}
\right.
\]
\end{theorem}

\begin{remark}
The well-covered dimensions of paths had already been computed in \cite{CY} using methods different from the one used in this paper.
\end{remark}

Now we look at the family of gear graphs. A gear graph over $2n+1$ vertices, denoted $G_{n}$ is the graph with vertex set $V(G_{n})=\{v_{0},...,v_{2n-1},v_{c}\}$ where: \\
\textbf{(a)} $v_{i}$ is adjacent to $v_{i-1\bmod  2n}$ and $v_{i+1\bmod 2n}$ for $0\leq i\leq 2n-1$. \\
\textbf{(b)}  if $i\in 2\mathbb{Z}$, then $v_{i}$ is adjacent to $v_{c}$.

We can compute the well-covered dimensions of the gear graphs using the same methods we used to compute the well-covered dimensions of the cycles.

\begin{corollary}
Let $G_n$ be the gear graph in $2n+1$ vertices, then
\[
wcdim(G_{n})= \left\{
\begin{array}{ll}
3 & $if$ \ n=3 \\
0 &  $if$  \ n>3
\end{array}
\right.
\]
\end{corollary}


We have used that if we had maximal independent sets that share many vertices then some relations between the weights of the vertices may be found. We close this section with a generic result that has some relation with the technique just mentioned.

\begin{lemma}\label{lemmazeroregions}
Let $G$ and $H$ be graphs such that $G$ is a subgraph of $H$, $wcdim(G, \textbf{F})=0$, and that there is a maximal independent set $M$ of $H\setminus G$ such that $M\cup N$ is a maximal independent set of $H$, for all maximal independent sets $N$ of $G$. Then, every well-covered weighting of $H$ (over $\textbf{F}$) is constant equal to zero on $V(G)$.
\end{lemma}

\begin{proof}
We look at the system created by considering the maximal independent sets of $H$ of the form $M\cup N$, where  $N$ is a maximal independent set of $G$.  This system yields no restrictions on the vertices of $H\setminus G$ but, since $wcdim(G, \textbf{F})=0$, we get that the weights for the vertices of $G$ must all be equal to zero. Since the equations in this system are a subset of  the equations in the system that we would need to analyze to get $wcdim(H,\textbf{F})$ then the result follows.
\end{proof}

\section{Blowups and lexicographic products}

In this section we look at the well-covered dimension of graphs that can be constructed from known ones by using various techniques. We begin with a definition.

\begin{definition}
Let $G$ be a graph and $t\in \mathbb{N}$. A $t$-blowup of a vertex $v_i\in V(G)$ is an independent set $V_{v_i} =\{v_{i1}, v_{i2}, \cdots , v_{it}  \}$ that `takes the place' of $v$. More precisely, wherever there was an edge joining $v$ to $w\in V(G)$ there is an edge joining $v_i$ with $w$.

The graph obtained by the $t$-blowup of $v$ will be denoted $G(tv)$. Similarly, for $v,w\in V(G)$ and $s,t \in \mathbb{N}$ we denote a `double blowup' $G(tv)(sw)$ as $G(tv , sw)$. For multiple blowups we extend in the natural way the notation set of double blowups.
\end{definition}

Note that $G(v)=G$ for all $v\in V(G)$.

\begin{lemma}\label{G(tv)}
Let $G$ be a graph with $V(G) = \{v_1, \cdots , v_n\}$, and $m = wcdim(G, \textbf F)$. Let $H= G(tv_1)$, where $t \in \mathbb{N}$. Then, $wcdim(H, \textbf F) = m+t-1$.
\end{lemma}

\begin{proof}
We begin by noticing that a maximal independent set of $G$ not containing $v_1$ is also a maximal independent set of $H$, and if $S = \{v_{1}, v_{i_2}, \cdots, v_{i_r}\}$ is a maximal independent set of $G$ then
\[
S' = \{v_{11}, \cdots , v_{1t}, v_{i_2},  \cdots , v_{i_r}\}
\]
is a maximal independent set of $H$. Moreover, it is easy to see that every maximal independent set of $H$ must be of one of these two types.

Let $M$ and $M(t)$ be the matrices associated to the systems of equations arising from looking for well-covered weightings of $G$ and $H$ respectively. We notice that $M(t)$ has $t-1$ more columns than $M$ but that it has exactly the same number of rows, and in fact the same rank as $M$, which is $n-m$. The result follows.
\end{proof}

By using this lemma repeatedly in a graph that is constructed from $G$ by a sequence of blowups of vertices of $G$ we get the following theorem.

\begin{theorem}
Let $G$ be a graph with $V(G) = \{v_1, \cdots , v_n\}$ and $m = wcdim(G, \textbf F)$. Let $H= G(t_1v_1, t_2v_2, \cdots , t_nv_n)$, where $t_i \in \mathbb{N}$ for all $i=1,2,\cdots ,n$. Then, 
\[
wcdim(H, \textbf F) = (m-n) + \sum_{i=1}^n t_i        
\]
\end{theorem}

Now we will look at the lexicographic product of graphs. We start with a definition.

\begin{definition}
The lexicographic product of $G$ and $H$, denoted $G \bullet H$, is the graph with vertex set $V(G)\times V(H)$ and edges joining $(g,h)$ and $(g',h')$ if and only if $gg' \in E(G)$ or $g=g'$ and $hh'\in E(H)$.
\end{definition}

\begin{corollary}\label{blowlex}
Let $G$ be a graph in $n$ vertices with $wcdim(G, \textbf F) =m$. Then, 
\[
wcdim\left(G  \bullet  \overline{K_t}, \textbf F \right)=m +n (t-1)
\]
where $t\in \mathbb{N}$.
\end{corollary}

\begin{proof}
Assume that $V(G)= \{v_1, \cdots , v_n\}$. The result follows from the previous theorem and the fact that  $G(tv_1, tv_2, \cdots , tv_n) \cong G  \bullet  \overline{K_t}$.
\end{proof}

The previous corollary is also a corollary of Theorem \ref{lex}. In order to prove this theorem we need a couple of linear algebra results that we will not prove, but will mention in full detail.

\begin{lemma}\label{linalg1}
Let $M$ be an $n\times m$ matrix and let $N$ be the $(n-1) \times m$ matrix obtained by subtracting the first row $R_1$ of $M$ from all the other rows of $M$, and then deleting $R_1$. Assume $rank(N)=k$, then, 
\[
rank(M) = \left\{
\begin{array}{ll}
k & $if $R_1$ is dependent of other rows of $M$$ \\
k+1 & $if $R_1$ is independent from other rows of $M$$ 
\end{array}
\right.
\]
\end{lemma}

For the next couple of results, we denote the Kronecker (or tensor) product of two matrices, $M$ and $A$, by $M\otimes A$.

\begin{remark}\label{linalg2}
Let $N, B$ and $C$ be matrices obtained by using the construction described in Lemma \ref{linalg1} from matrices $M, A$, and $M\otimes A$ respectively, we will re-arrange rows in the matrices if necessary to get, when possible, the first row to be dependent of the others. Then, $rank(C)=rank(M\otimes A)$ whenever there is a row that is dependent of others in $A$ or $M$, as in these cases we can always choose a row of $M\otimes A$ that depends on the other rows of this matrix. On the other hand, if both $M$ and $A$ have linearly independent rows, then $M\otimes A$ also has this property (using $rank(A\otimes M) = rank(A)rank(M)$), and thus $rank(C)=rank(M\otimes A)-1$.

If we now use Lemma  \ref{linalg1}, and assume $rank(N)=k$ and $rank(B)=q$, then
\[
rank(C) = \left\{
\begin{array}{ll}
kq & $if both $M$ and $A$ have linearly dependent rows$ \\
k(q+1) & $if $M$ has linearly dependent rows and $A$ does not$ \\
(k+1)q & $if $A$ has linearly dependent rows and $M$ does not$ \\
(k+1)(q+1)-1 & $if both $M$ and $A$ have linearly independent rows$
\end{array}
\right.
\]
\end{remark}

Now we have all the tools needed to prove.

\begin{theorem}\label{lex}
Let $G$ and $H$ be graphs with $|V(G)|= a$, $|V(H)|= b$, $wcdim(G, \textbf F)=n$, and $wcdim(H, \textbf F)=m$. Then, 
\[
wcdim(G\bullet H, \textbf F) = nb+am - nm +\delta_{m-b+1, i}(n-a) +\delta_{n-a+1, j}(m- b)
\]
where $\delta_{xy}$ represents the Kronecker delta, and $i$, $j$ are the number of maximal independent sets of $H$ and $G$ respectively.
\end{theorem}

\begin{proof}
We first notice that if $S = \{v_{1_1}, v_{i_2}, \cdots, v_{i_r}\}$ is a maximal independent set of $G$ then
\[
S' = \{w_{1i_1}, \cdots , w_{t_1i_1}, w_{1i_2}, \cdots , w_{t_2i_2},  \cdots , w_{1i_r}, \cdots , w_{t_ri_r}\}
\]
is a maximal independent set of $G\bullet H$, where $\{  w_{1i_j}, \cdots , w_{t_ji_j}  \}$ is a maximal independent set of $H$ for all $j=1,2, \cdots , r$. Moreover, it is easy to see that every maximal independent set of $G\bullet H$ must be obtained this way. 

Set the weight-sums of each of the independent sets of $G$ equal to zero. Let $M$ be the matrix representing that homogeneous system of equations. Note that the matrix $N$ needed to find $wcdim(G, \textbf F)$ is obtained from $M$ by using the construction described in Lemma \ref{linalg1}.  Similarly, by repeating this process with $H$ we obtain $B$ (needed for finding $wcdim(H, \textbf F)$) out of $A$ (found by setting the weight-sums of the maximal independent sets of $H$ equal to zero).

Now we notice that (because of the first paragraph in this proof) $A\otimes M$ is the matrix associated to the homogeneous system given by setting the weight-sums of all the independent sets of $G\bullet H$ equal to zero. It follows that we are interested in finding the rank of the matrix $C$ obtained from $A\otimes M$ by using the construction described in Lemma \ref{linalg1}.

Since $rank(N) = a-n$, $rank(B)=b-m$, and $|V(G\bullet H)|=ab$, then using that a matrix has linearly dependent rows if and only if its rank is not equal to its number of rows, and Remark \ref{linalg2}, we get
{\small
\[
wcdim(G\bullet H, \textbf F) = \left\{
\begin{array}{ll}
nb+am - nm & $if $i\neq b-m+1, \ j\neq a-n+1$$ \\
nb+am- nm+n-a & $if $i=b-m+1, \ j\neq  a-n+1$$ \\
nb+am - nm+m-b & $if $i\neq b-m+1, \ j=  a-n+1$$ \\
nb+am - nm +m-b +n- a  & $if $i=b-m+1, \ j=  a-n+1$$
\end{array}
\right.
\]
}
\hspace{-.09in} where $i$, $j$ represent the number of maximal independent sets of $H$ and $G$ respectively (which are the number of rows of $A$ and $M$ respectively).

The result follows from the definition of the Kronecker delta.
\end{proof}

\begin{corollary}
Let $G$ and $H$ be graphs with more maximal independent sets than vertices, and such that $|V(G)|= a$, $|V(H)|= b$, $wcdim(G, \textbf F)=n$ and $wcdim(H, \textbf F)=m$. Then, 
\[
wcdim(G\bullet H, \textbf F) = nb+am - nm
\]
\end{corollary}

\begin{remark}
As mentioned above, Corollary \ref{blowlex} is a corollary of Theorem \ref{lex}. In order to see this we just need to notice that $\overline{K_t}$ has one maximal independent set and that $wcdim(\overline{K_t}) = V(\overline{K_t})=t$.
\end{remark}

\section{Cartesian products: Paths and Cycles.}

The Cartesian product of $G\times H$ is the graph with vertices $(u,v)$ where $u\in V(G)$ and $v\in V(H)$ and there exists an edge joining $(u_1,v_1)$ with $(u_2,v_2)$ iff there exists an edge in $G$ joining $u_1$ and $v_2$ and $v_1=v_2$ or there exists an edge in $H$ joining $v_1$ and $v_2$ and $u_1=u_2$. 

We start by exploring products of paths and/or cycles. Just like in Section \ref{sectionbasic} we will study these two classes of graphs almost simultaneously. We will also borrow from that section the idea of  comparing maximal independent sets that agree in all but a few of their vertices.  Also, given that we will use many pictures in this section we need the following definitions.

\begin{definition}
If $G = P_n \times P_m$ or $G=P_n\times C_m$, where $n,m>1$,  we will say that a vertex is on the interior of $G$ (or that it is an internal vertex) if its degree is equal to $4$. A vertex of degree $2$ or $3$ will be said to be on the boundary of $G$ (or that it is a boundary vertex). A vertex of degree $2$ will be also called a corner of $G$.
\end{definition}

\begin{lemma}\label{lembdryweit}
Let $w$ be any well-covered weighting of the graph $G$.
\begin{enumerate}
\item If $G = P_n \times P_m$, where $m\geq 2$ and $n\geq 4$, or
\item If $G=P_n\times C_3$, where $n\geq 2$, or
\item If $G=P_n\times C_m$, where $m\geq 5$ and $n\geq 2$,
\end{enumerate}
then any two adjacent boundary non-corner vertices, $a$ and $b$, of $G$ must satisfy $w(a)=w(b)$.
\end{lemma}

\begin{proof}
Let $a$ and $b$ be two adjacent boundary non-corner vertices of $G$. Consider the following two maximal independent sets of $P_4\times P_2$. 

\begin{figure}[h]
\centering
\includegraphics[height=.6in]{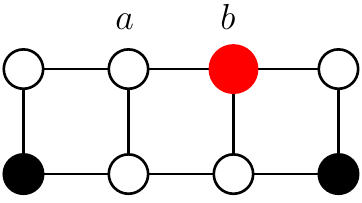}    \hspace{.7in} \includegraphics[height=.6in]{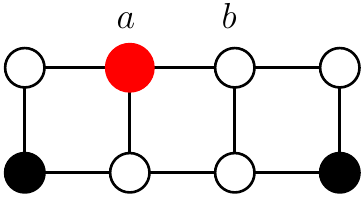}   
\end{figure}

Here we can see that $w(a)=w(b)$, for all well-covered weightings of $P_4\times P_2$. Using a greedy algorithm we can find maximal independent sets of $G$ that are the same except at these `local' pictures. This is immediate for cases (1) and (3). On the other hand, for case (2) we need (in each of these figures) to identify the vertices on their leftmost `side' to those on their rightmost `side' to create a local picture of $P_n\times C_3$. The result follows.
\end{proof}

Now that we know how non-corner boundary vertices behave we need to take a look at corner vertices.

\begin{lemma}\label{lemcrnrweit} 
Let $G = P_n \times P_m$, where $n,m\geq 3$, and $w$ be any well-covered weighting of $G$. Let $b$ and $c$ be two (boundary) vertices adjacent to a corner vertex $a$. Then, $w(a)=w(b)=w(c)$.
\end{lemma}

\begin{proof}
Let $a$ be a corner vertex and $b$ one of its neighbors. We consider the two local pictures of this corner below

\begin{figure}[h]
\centering
\includegraphics[height=.7in]{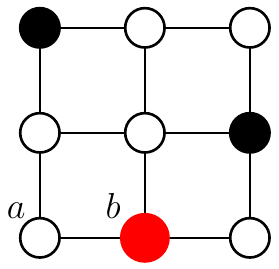}    \hspace{.5in} \includegraphics[height=.7in]{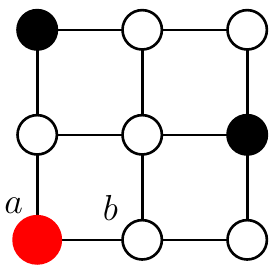}   
\end{figure}

Using the same ideas used in the proof of Lemma \ref{lembdryweit}, we get that these pictures imply that $w(a)=w(b)$, for all well-covered weighting of $G$. Finally, by simply reflecting the picture we obtain $w(a)=w(c)$.
\end{proof}

We summarize the results in Lemmas \ref{lembdryweit} and \ref{lemcrnrweit} in the following corollary. 

\begin{corollary}\label{corPnPmPnCm}
Let $w$ be any well-covered weighting of the graph $G$.
\begin{enumerate}
\item If $G = P_n \times P_m$, where $n,m\geq 4$, then $w(a)=w(b)$, for any two boundary vertices, $a$ and $b$, of $G$.
\item If $G= P_n\times C_m$, where $n\geq 2$, and $m= 3$ or $m\geq 5$, then $w(a)=w(b)$, for any two boundary vertices, $a$ and $b$, that are on the same `boundary cycle' of $G$.
\end{enumerate}
\end{corollary}

Now that we understand boundary vertices we move on to study interior vertices.

\begin{lemma}\label{lemintweit}
Let $w$ be any well-covered weighting of the graph $G$. 
\begin{enumerate}
\item If $G = P_n \times P_m$, where $m, n\geq 5$, or
\item If $G=P_n\times C_m$, where $m\geq 6$ and $n\geq 5$, or
\item If $G=C_n\times C_m$, where $m, n\geq 6$,
\end{enumerate}
then $w(a)=0$, for all $a\in E(G)$.
\end{lemma}

\begin{proof}
In order to prove this we will embed a $P_5\times P_5$ and a $P_4\times P_4$ as `local pictures', hence the bounds for $m$ and $n$ in the (three different) hypothesis. 

We consider the following figures, which show maximal independent sets of $P_5\times P_5$ and $P_4\times P_4$ that share all but the grey vertices (red if you read this in color)

\begin{figure}[h]
\centering
\includegraphics[height=1in]{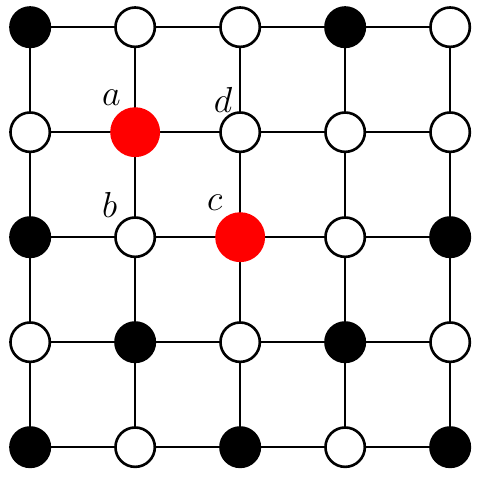}    \hspace{.2in} \includegraphics[height=1in]{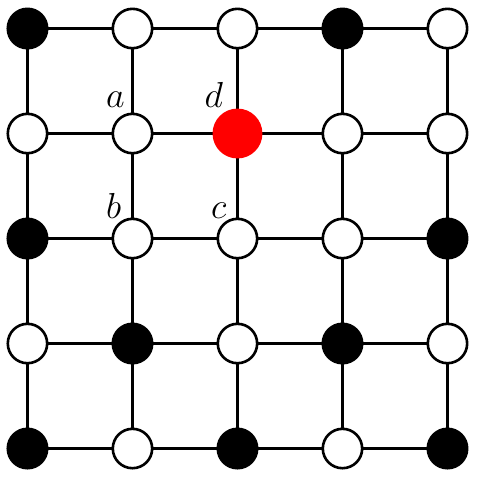}  \hspace{.3in}  \includegraphics[height=1in]{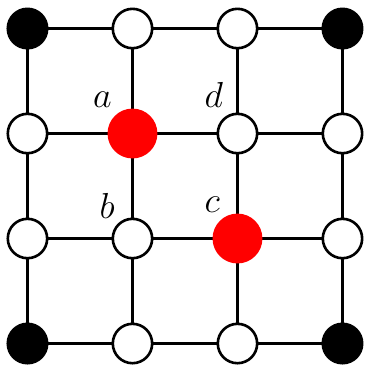}    \hspace{.2in} \includegraphics[height=1in]{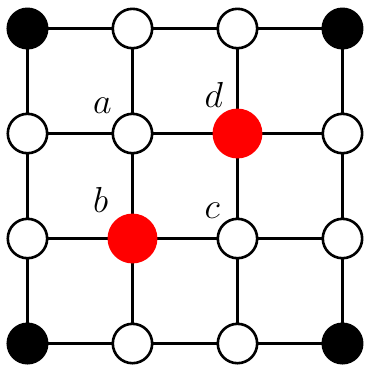} 
\end{figure}

The picture on the left (where we consider $P_5\times P_5$ as a subgraph of $G$) implies that $w(a)+w(c) =w(d)$, for all well-covered weightings of $G$. However, by considering the picture on the right (where we consider $P_4\times P_4$ as a subgraph of $G$) we obtain $w(a)+w(c) =w(b)+w(d)$, for all well-covered weightings of $G$. It follows that $w(b)=0$, for all well-covered weightings of $G$. 

By embedding reflected and/or rotated versions of these figures into $G$ we obtain the desired result.
\end{proof}

Since all vertices in $C_n\times C_m$ are interior, the following result holds trivially.

\begin{lemma}\label{thmCnCm}
Let $m,n\in \mathbb{N}$ such that $m,n\geq 6$, then $wcdim(C_n\times C_m, \textbf{F})=0$, for all fields $\textbf{F}$.
\end{lemma}

We will now get results similar to Lemma \ref{thmCnCm}, but for $P_n\times P_m$ and $P_n\times C_m$.

\begin{lemma}\label{thmPnCm&PnPm}
Let $m,n\in \mathbb{N}$.
\begin{enumerate}
\item  If $m, n\geq 5$, then $wcdim(P_n \times P_m, \textbf{F})=0$, for all fields $\textbf{F}$.
\item If $m\geq 6$ and $n\geq 5$, then $wcdim(P_n\times C_m, \textbf{F})=0$, for all fields $\textbf{F}$.
\end{enumerate}
\end{lemma}

\begin{proof}
We consider the following figures, where $b$ is a boundary vertex.

\begin{figure}[h]
\centering
\includegraphics[height=1in]{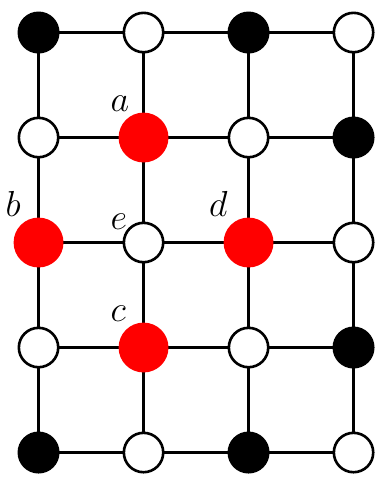}    \hspace{.7in} \includegraphics[height=1in]{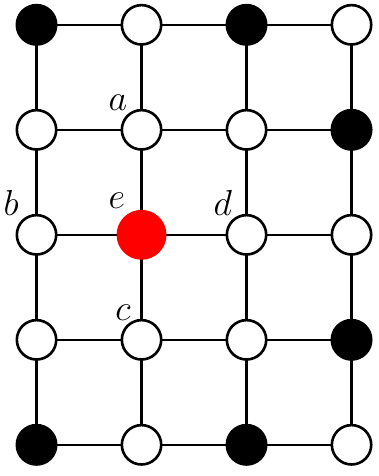}    
\end{figure}

Given the bounds for $m$ and $n$, we may embed these figures into $G$ for either case. It follows that $w(e) = w(a)+w(b)+w(c)+w(d)$, for all well-covered weightings of $G$. However, Lemma \ref{lemintweit} says that $w(a)=w(c)=w(d)=w(e) = 0$. Since $b$ can be chosen to be on any `boundary cycle' of $G$ the result follows from Lemma \ref{corPnPmPnCm}.
\end{proof}

Before phrasing our main theorem for this section (Theorem \ref{thmmainP_nxC_m})  we add one more result, which shows that the lower bound for the size of a cycle in Lemma \ref{thmPnCm&PnPm} is sharp.

\begin{lemma}
$wcdim(P_n\times C_5, \textbf{F})=2$, for all fields $\textbf{F}$, and all $n\in \mathbb{N}$ such that $n\geq 6$.
\end{lemma}

\begin{proof}
Let $G= P_n\times C_5$, and $w$ be a well-covered weighting of $G$. 

As we have done before, we will use pictures to find relations between the weights of vertices of $G$. Hence, we consider

\begin{figure}[h]
\centering
\includegraphics[height=1in]{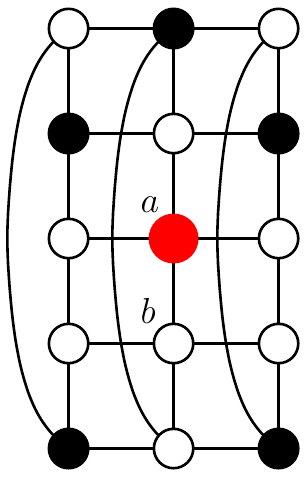}    \hspace{.4in} \includegraphics[height=1in]{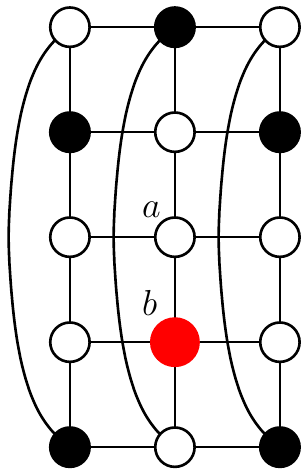}    \hspace{1.1in}       \includegraphics[height=1in]{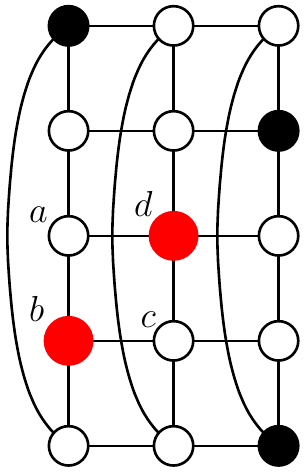}    \hspace{.4in} \includegraphics[height=1in]{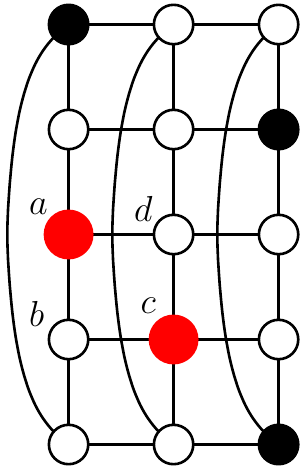}    
\end{figure}

Note that we are assuming the the left cycle of the figure on the right is a `boundary cycle' of $G$.

It is easy to see that the figure on the left implies $w(a)=w(b)$, for any two vertices, $a$ and $b$, of $G$ that are on any given non-boundary cycle. We use this to realize that in the figure on the right we get $w(c)=w(d)$. Moreover, since $w(a)+w(c)=w(b)+w(d)$ then $w(a)=w(b)$. Hence,  any two vertices on any given cycle have the same weight.

Next we consider yet two more pairs of figures.

\begin{figure}[h]
\centering
\includegraphics[height=.9in]{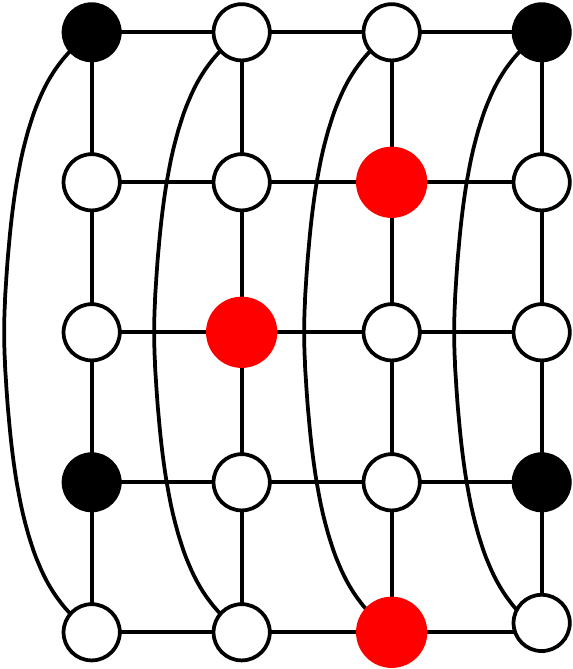}    \hspace{.05in} \includegraphics[height=.9in]{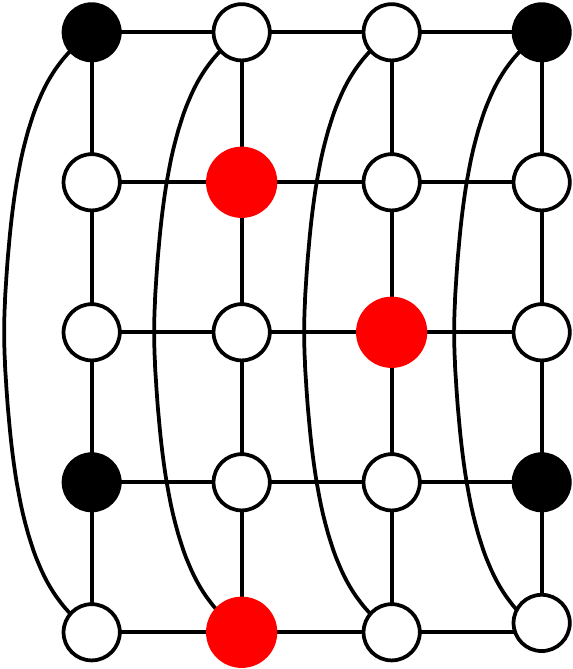}  \hspace{.7in}  \includegraphics[height=.9in]{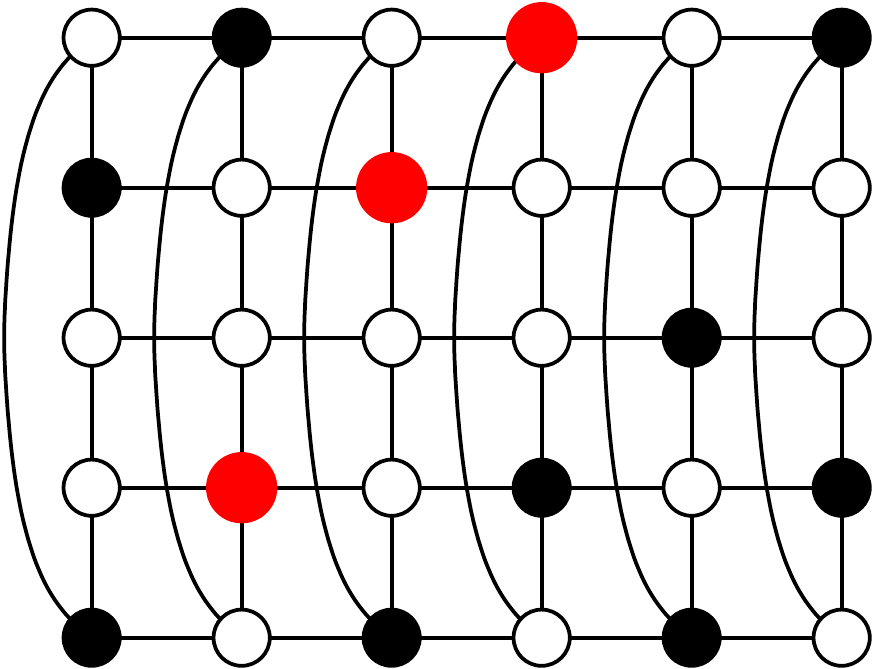}    \hspace{.05in} \includegraphics[height=.9in]{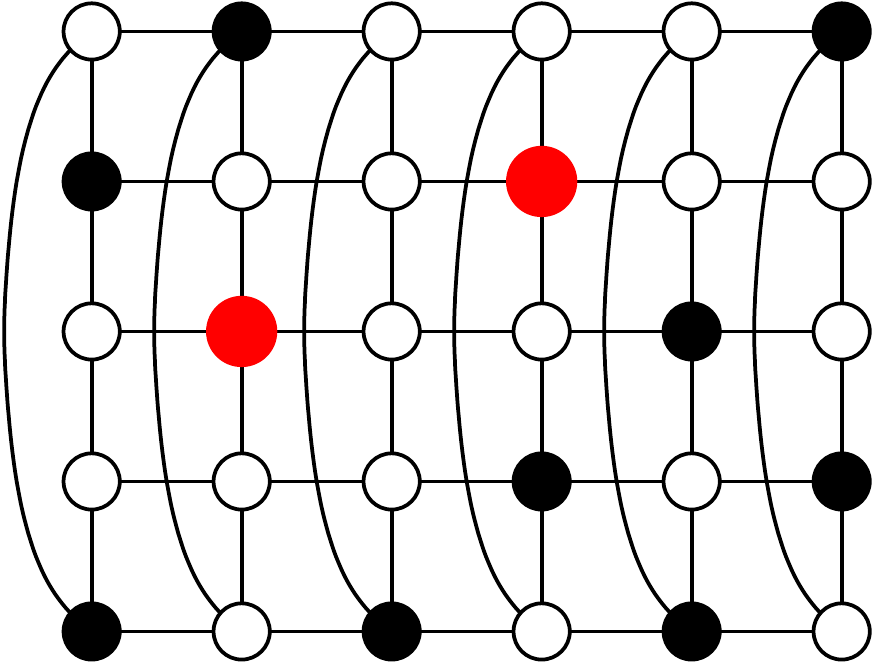} 
\end{figure}

The figure on the left assures that every non-boundary vertex has the same weight. Now we use this to obtain, from the figure on the right, that all the non-boundary vertices of $G$ have weight equal to zero.

So far, we know that $wcdim(P_n\times C_5, \textbf{F})\leq 2$. We notice that a boundary cycle of $G$ must contain exactly two vertices of every maximal independent set of $P_n\times C_5$. Hence, the functions defined by
\[
f_1(v) = \left\{ \begin{array}{ll}
1 \ \ \ & \text{if $v$ is a vertex of $\mathcal{C}_1$} \\ 
0 \ \ \ & \text{otherwise} \\ 
\end{array}
\right.
\]
\[
f_2(v) = \left\{ \begin{array}{ll}
1 \ \ \ & \text{if $v$ is a vertex of $\mathcal{C}_n$} \\ 
0 \ \ \ & \text{otherwise} \\ 
\end{array}
\right.
\]
where $\mathcal{C}_1$ and $\mathcal{C}_n$ are the two boundary cycles of $P_n\times C_5$, are linearly independent well-covered weightings of $P_n\times C_5$.
\end{proof}

We close this section with a theorem that collects all results in this section. 

\begin{theorem}\label{thmmainP_nxC_m} 
Let $m,n\in \mathbb{N}$
\begin{enumerate}
\item If $m,n\geq 6$, then $wcdim(C_n\times C_m, \textbf{F})=0$, for all fields $\textbf{F}$. 
\item  If $m, n\geq 5$, then $wcdim(P_n \times P_m, \textbf{F})=0$, for all fields $\textbf{F}$. 
\item If $m\geq 6$ and $n\geq 5$, then $wcdim(P_n\times C_m, \textbf{F})=0$, for all fields $\textbf{F}$. 
\item  If $n\geq 6$, then $wcdim(P_n\times C_5, \textbf{F})=2$, for all fields $\textbf{F}$.
\end{enumerate} 
\end{theorem}

\begin{remark}\label{remaddendum}
Theorem \ref{thmmainP_nxC_m} may be improved for `small' $m$ and $n$. In these cases, the values of $wcdim(G)$ are found using similar techniques and the proofs are heavy in pictures and/or matrix computations. For this reason we will not include these results here, they may be found on \verb+arxiv.org+,  or the fifth author's web site. These results will not be published otherwise.
\end{remark}

\section{Cartesian Products with Complete Graphs}

The greedy independent decomposition of a graph provides a means to construct maximal independent sets of $G\times H$. The following definition and lemma may be found in  \cite{O}. 

\begin{definition}
 Let $I_1$ be a maximal independent set of G. Choose $I_2,I_3\cdots$ such that $I_k$ is a maximal independent set in $G\backslash(I_1\cup I_2\cup\cdots\cup I_{k-1})$. Then, $\{I_j\}_{i=1}^n$ is a greedy independent decomposition of $G$. 
 \end{definition}
 
 The following result will be used later in this article. 

\begin{lemma}[Ovetsky, \cite{O}]
Let $\{I_j\}_{i=1}^n$ and $\{J_j\}_{j=1}^n$ be greedy independent decompositions of graphs G and H respectively, and without loss of generality suppose $n\leq m$. Then $(I_1\times J_1)\cup (I_2\times J_2)\cup\cdots\cup(I_n\times J_n)$ is a maximal independent set in $G\times H$. 
\end{lemma}


\begin{theorem}\label{thm5.1}
Let $H$ be a graph such that there exists a greedy independent decomposition of cardinality c. Then, $wcdim(K_n\times H, \textbf{F})\leq |V(H)|$ where $c<n$. 
\end{theorem}

\begin{proof}
Let $w_j$ denote the vertices of $H$, where $j=1,2,\ldots, m$, and $v_i$ denote the vertices of $K_n$, where $i=1,2,\ldots,n$. Let $\{S_i\}_{i=1}^{p}$ be a greedy independent decomposition of $H$ such that $|S|=p<n$. Let $\{J_i\}_{i=1}^{n}$ be a greedy independent decomposition of $K_n$ such that $J_i=\{v_i\}$. By the greedy independent theorem, 
\[
A= \bigcup_{i=1}^{p} (J_i\times S_i)
\]
is a maximal independent set of $K_n\times H$. Let $w_j\in S_1$  for some $j=1,2,\ldots,m$. Then, $(v_1,w_j)\in A$. Switch the vertex at $(v_1,w_j)$ to $(v_n,w_j)$ to create a new maximal independent set. Hence, $w(v_1,w_j)=w(v_n,w_j)$ for some well-covered weighting $w:V(K_n\times H)\rightarrow \textbf{F}$. Then, switch each vertex of the form $(v_1,w_j)$ to $(v_n,w_j)$ for every $w_j\in S_1$. Hence,  $w(v_1,w_j)=w(v_n.w_j)$ for all $w_j$ such that $w_j\in S_1$. 

Set $a=1$. Let $w_j\in S_{a+1}$  for some $j=1,2,\ldots,m$. Then, $(v_{a+1},w_j)\in A$. Switch the vertex at $(v_{a+1},w_j)$ to $(v_a,w_j)$ to create a new maximal independent set. Hence, $w(v_{a+1},w_j)=w(v_a,w_j)$. Repeat for every $w_j\in S_{a+1}$. Hence, $w(v_{a+1},w_j)=w(v_a,w_j)$ for every $w_j\in S_{a+1}$.
Set $a=a+1$ and repeat for every $a\leq p-1$. 

After the above process has been completed, $A$ has been shifted vertically by exactly one vertex to create a new maximal independent set $A^{'}$. Continue to shift the maximal independent set vertically in the same manner until $A^{'} = A$. Then, you will have shown that $w(v_i,w_j)=w(v_k,w_j)$ for every $i,k=1,2,\ldots,n$. Hence, $wcdim(K_n\times H,\textbf{F})\leq |V(H)|$. 
\end{proof}

\begin{lemma}\label{lemma5.1}
For every $P_n$, there is a greedy independent decomposition of cardinality exactly $2$.
\end{lemma}

\begin{proof}
Let $J=\{J_i\}_{i=1}^p$ be a greedy independent decomposition for $P_n$. Set $J_1=\{v_i: i=2n-1 \hspace{.05in} \mbox{for some} \hspace{.05in} n\in\mathbb{N} \}$. Then, $J_2=\{v_i: i=2n \hspace{.05in} \mbox{for some} \hspace{.05in} n\in\mathbb{N}\}$. Hence, $|J|=2$. 
\end{proof}

\begin{corollary}
$wcdim(K_n\times P_m, \textbf{F})$ for all $n,m\geq 3$. 
\end{corollary}

\begin{proof} By Lemma \ref{lemma5.1} and Theorem \ref{thm5.1}, $wcdim(K_n\times P_m, \textbf{F}) \leq |V(P_m)|=m$. Let $f_i:V(K_n\times P_m)\rightarrow \textbf{F}$ where
\[
f_i=
\left\{
\begin{array}{cl}
1 & \mbox{for all vertices of column $i$} \\
0 & \mbox{otherwise} \\
\end{array}
\right.
\]
Note that $\sum_{v\in A}f_i(v) = 1$ for every $i$ and every maximal independent set $A$. Hence, $F=\{f_i\}_{i=1}^m$ is a linearly independent set of well-covered weightings. Hence, $wcdim(K_n\times P_m,\textbf{F}) \geq m$ because $|F|=m$. Therefore, $wcdim(K_n\times P_m, \textbf{F})=m$.
\end{proof}

\begin{lemma}\label{lemma5.3}
If $m$ is even, there is a greedy independent decomposition of $C_m$ of cardinality exactly 2. If $m$ is odd, there is a greedy independent decomposition of $C_m$ of cardinality exactly 3. 
\end{lemma}

\begin{proof}
Let $J=\{J_i\}_{i=1}^{p}$ be a greedy independent decomposition of a cycle. 

\noindent  \textbf{Case 1: } Let $m$ be even. 

Set $J_1=\{v_i:i=2n-1 \hspace{.05in} \mbox{for some} \hspace{.05in} n\in\mathbb{N}\}$. Then, $J_2=\{v_i:i=2n \hspace{.05in} \mbox{for some} \hspace{.05in} n\in\mathbb{N}\}$. Hence, $|J|=2$. 

\noindent  \textbf{Case 2: } Let $m$ be odd.

Set $J_1=\{v_i:i=2n-1 \hspace{.05in} \mbox{for some} \hspace{.05in} n\in\mathbb{N}\}\backslash\{v_m\}$. Then, $J_2=\{v_i:i=2n \hspace{.05in} \mbox{for some} \hspace{.05in} n\in\mathbb{N}\}$ and $J_3=\{v_m\}$. Therefore, $|J|=3$. 
\end{proof}

\begin{corollary}
$wcdim(K_n\times C_m, \textbf{F})$ for all $n,m\geq 4$. 
\end{corollary}

\begin{proof} By Lemma \ref{lemma5.3} and Theorem \ref{thm5.1}, $wcdim(K_n\times P_m, \textbf{F}) \leq |V(C_m)|=m$. Let $f_i:V(K_n\times C_m)\rightarrow \textbf{F}$ where
\[
f_i=
\left\{
\begin{array}{cl}
1 & \mbox{for all vertices of column $i$} \\
0 & \mbox{otherwise} \\
\end{array}
\right.
\]
Then, $F=\{f_i\}_{i=1}^m$ is a set of  linearly independent well-covered weightings of $K_n\times C_m$. Hence, $wcdim(K_n\times C_m)\geq m$. Therefore, $wcdim(K_n\times C_m, \textbf{F})=m$. 
\end{proof}

\begin{corollary}
$wcdim(K_n\times K_m,\textbf{F})=m$ for all $n>m\geq 3$. 
\end{corollary}

\begin{proof} Let $J=\{J_i\}_{i=1}^m$ be a greedy independent decomposition of $K_m$. Hence, $|J|=m<n$. By Theorem \ref{thm5.1}, $wcdim(K_n\times K_m)\leq |V(K_m)|=m$. Let $f_i:V(K_n\times K_m)\rightarrow \textbf{F}$ where
\[
f_i=
\left\{
\begin{array}{cl}
1 & \mbox{for all vertices of column $i$} \\
0 & \mbox{otherwise} \\
\end{array}
\right.
\]
Then, $F=\{f_i\}_{i=1}^m$ is a set of linearly independent well-covered weightings of $K_n\times K_m$. Hence, $wcdim(K_n\times K_m) \geq m$. Therefore, $wcdim(K_n\times K_m,\textbf{F})=m$.  
\end{proof}

\begin{theorem}
$wcdim(K_n\times K_n,\textbf{F})=2n-1$ for all $n\geq 3$. 
\end{theorem}

\begin{proof} Let $A$ be a maximal independent set of $K_n\times K_n$ such that $(v_i,w_j),(v_k,w_l)\in A$ where $i\neq k$ and $j\neq l$. We can create a new maximal independent set $B$ such that $(v_i,w_l),(v_k,w_j)\in B$ and $A\backslash\{(v_i,w_j),(v_k,w_l)\}=B\backslash\{(v_i,w_l),(v_k,w_j)\}$. Hence, $w(v_i,w_j)+w(v_k,w_l)=w(v_i,w_l)+w(v_k,w_j)$ for any well-covered weighting $w$. Therefore, $w(v_i,w_j)=w(v_i,w_l)+w(v_k,w_j) - w(v_k,w_l)$ for any $i,j,k,l$ such that $i\neq k$ and $j\neq l$.

Choose row $k$ and column $l$ and assign arbitrary weights to the vertices of row $k$ and column $l$. Because each row and column have $n$ vertices and the row and column share the vertex $(v_k,w_l)$, there are $2n-1$ arbitrary weights assigned. Hence, for any vertex $(v_i,w_j)$ such that $i\neq k$ and $j\neq l$, $w(v_i,w_j)=w(v_i,w_l)+w(v_k,w_j)-w(v_k,w_l)$. Therefore, $wcdim(K_n\times K_n) \leq 2n-1$. Let $f_i,f_j:V(K_n\times K_n)\rightarrow \textbf{F}$ where
\[
f_i=
\left\{
\begin{array}{cl}
1 & \mbox{for all vertices of column $i$} \\
0 & \mbox{otherwise} \\
\end{array}
\right.
\]
and
\[
f_j=
\left\{
\begin{array}{cl}
1 & \mbox{for all vertices of row $j$} \\
0 & \mbox{otherwise} \\
\end{array}
\right.
\]
Hence, $f_i$ and $f_j$ are well-covered weightings of $K_n\times K_n$. Let $F=\{\{f_i\}_{i=1}^n,\{f_j\}_{j=1}^n\}$. Note that $F$ is not a linearly independent set. However, $F^*=F\backslash\{f_{j=n}\}$ is a linearly independent set and $|F^*|=2n-1$. Hence, we can conclude that $wcdim(K_n\times K_n,\textbf{F}) \geq 2n-1$. Therefore, $wcdim(K_n\times K_n,\textbf{F}) = 2n-1.$
\end{proof}



\end{document}